\documentclass{amsart}

\usepackage{amssymb}
\usepackage{centernot}
\usepackage[shortlabels]{enumitem}
\usepackage{tikz}
\usepackage{subfigure, float}
\usepackage{pgfplots}
\usepackage{comment}
\usepackage{mathrsfs}
\usepackage{hyperref}
\hypersetup{colorlinks,linkcolor={red},citecolor={olive},urlcolor={red}}
\usepackage{mathtools}
\mathtoolsset{showonlyrefs}
\usepackage{mathtools}
\mathtoolsset{showonlyrefs}

\usepackage{theoremref}

\newtheorem{theorem}{Theorem}
\newtheorem{lemma}[theorem]{Lemma}

\newtheorem{question}{Question}

\theoremstyle{definition}

\newcommand{\ind}[1]{\mathbf 1 \{#1\} }
\newcommand{\f}{\frac}

\def\E{\mathbb{E}}
\def\P{\mathbb{P}}

\def\V{\mathbb{V}}
\newcommand{\T}{\mathcal T}
\newcommand{\Ts}{\mathcal T^*}
\def\root{\mathbf{0}}
\newcommand{\0}{\mathbf{0}}
\def\Aut{\textup{Aut}}

\def\car{\mathfrak{a}}
\def\spot{\mathfrak{b}}
\renewcommand{\a}{\car}
\renewcommand{\b}{\spot}
\newcommand{\lrl}{\longleftrightarrow}

\newcommand{\lrljk}{\overset{j,k}\longleftrightarrow}
\newcommand{\nlrl}{\centernot \lrl}
\newcommand{\so}{\spot_\0}
\newcommand{\co}{\car_\0}
\newcommand{\Eso}{\E[S(\so)]}
\newcommand{\Eco}{\E[S(\co)]}
\newcommand{\zo}{\zeta_t^{\0}}
\renewcommand{\emptyset}{\varnothing}
\newcommand{\OR}{\texttt{OR}}
\newcommand{\XOR}{\texttt{XOR}}

\title{Diffusion-limited annihilating-coalescing systems}

	\author[S.~Ahn]{Sungwon Ahn}
	\email{\texttt{sahn02@roosevelt.edu}}

 \author[M.~Junge]{Matthew Junge}
\email{\texttt{Matthew.Junge@baruch.cuny.edu}}

	\author[H.~Lyu]{Hanbaek Lyu}
	\email{\texttt{hlyu@math.wisc.edu}}

 	\author[L.~Reeves]{Lily Reeves}
	\email{\texttt{zw477@cornell.edu}}
	
	\author[J.~Richey]{Jacob Richey}
	\email{\texttt{jfrichey@math.ubc.ca}}

	\author[D.~Sivakoff]{David Sivakoff}

	\email{\texttt{dsivakoff@stat.osu.edu}}

	\thanks{Part of this research was completed during the 2019 AMS Mathematical Research Community in Stochastic Spatial Systems. Junge was partially supported by NSF grant DMS-2115936. Lyu was partially supported by DMS-2206296 and DMS-2010035. Lyu and Sivakoff were partially supported by the NSF grant CCF--1740761.}

\begin{document}
\maketitle

	\begin{abstract}
We study a family of interacting particle systems with annihilating and coalescing reactions. Two types of particles are interspersed throughout a transitive unimodular graph. Both types diffuse as simple random walks with possibly different jump rates. Upon colliding, like particles coalesce up to some cap and unlike particles annihilate. We describe a phase transition as the initial particle density is varied and provide estimates for the expected occupation time of the root. For the symmetric setting with no cap on coalescence, we prove that the limiting occupation probability of the root is asymptotic to $2/3$ the occupation probability for classical coalescing random walk. This addresses an open problem from Stephenson.
	\end{abstract}

	\maketitle

	\section{Introduction}
	\label{sec:intro}
	
	Coalescing and annihilating random walk were among the first interacting particle systems to be rigorously studied \cite{griffeath, bramson1980asymptotics, arratia}. Both processes start with particles placed throughout the integer lattice $\mathbb Z^d$. These particles simultaneously perform independent continuous time simple random walk. In \emph{coalescing random walk}, when two particles meet, they coalesce. In \emph{annihilating random walk}, they annihilate. The main interest was understanding the limiting density of particles. See \eqref{eq:crwa} for these asymptotics.
 
    Two-type \emph{diffusion-limited annihilating systems} (DLAS) were investigated by Bramson and Lebowitz \cite{bramson1991asymptotic}. They considered systems in which $A$- and $B$-particles are initially distributed throughout $\mathbb Z^d$ according to independent Poisson fields. Both particle types move as independent continuous time simple random walks with the same jump rate. When particles of opposite type meet, they mutually annihilate. The one-to-one nature of annihilating reactions makes such systems critical when the two Poisson fields have the same intensity. Bramson and Lebowitz settled conflicting predictions from the physics community by working out the exponents of the limiting particle density for the sub-critical and critical regimes. These quantities were found to exhibit anomalous non-mean-field behavior in dimension $d \leq 3$ \cite{bramson1991asymptotic, bramson1991spatial, lee1995renormalization}. For example, at criticality the probability a particle is at the origin at time $t$ was found to be on the order of $t^{-d/4}$ for $d \leq 3$ and $t^{-1}$ for $d \geq 4$ \cite{bramson1991asymptotic}.

     Many of the arguments used by Bramson and Lebowitz depended on particles jumping at the same speed. There has been a resurgence in interest in developing more robust tools for analyzing asymmetric variants. In \cite{cabezas2019recurrence}, Cabezas, Rolla, and Sidoravicius studied DLAS in which $A$- and $B$-particles may have different jump rates. They prove that such systems on certain transitive unimodular graphs visit the root infinitely often and the probability a particle is at the origin at time $t$ is asymptotically at least $t^{-1}$.  Related results were proven in \cite{damron2019parking} for the \emph{parking process} in which $B$-particles are stationary (also considered in \cite{cabezas2019recurrence}). More in-depth studies of the particle density took place in \cite{parking_on_integers, johnson2020particle, damron2021stretched}. Significant progress has been made for the process with stationary $B$-particles on $\mathbb Z$. However, much remains unknown. For example, the exponent for the probability an $A$-particle is at the root at time $t$ is not known for any DLAS with asymmetric jump rates on $\mathbb Z^d$ for $d \geq 2$. Bahl, Barnet, Johnshon and Junge proved a general monotonicity condition for DLAS as the volatility of the initial configuration is increased \cite{bahl2022diffusion}.  There is also interest in asymmetric and coalescing ballistic annihilating systems \cite{junge2022phase, benitez2023three}.

    Stephenson studied a one-type model in which both annihilating and coalescing reactions may occur \cite{stephenson1999asymptotic}. The process takes place on $\mathbb Z^d$ with a particle initially at every site. Each particle follows an independent simple random walk. There is a threshold $n$, which is the minimum number of particles that must be at a site for a reaction to occur. When a particle moves to a site containing $i \geq n-1$ particles, all of the particles  mutually annihilate with probability $a_i$ and coalesce with probability $c_i$. Stephenson proved that the probability the origin is occupied at time $t$ is asymptotic to $t^{-1/(n-1)}$. For technical reasons, this result required the additional hypotheses that $n \geq 3$, $d \geq 2n +4$, and both $a_i$ and $a_i + c_i$ are increasing. For the model with no annihilation ($a_i=0$), Van den Berg and Kesten, over the course of two articles  \cite{van2000asymptotic,van2002randomly}, proved a more robust version of Stephenson's theorem that holds for all $n \geq 2$ and $d \geq 3$.

    Drawing inspiration from the work of Stephenson, Van den Berg, and Kesten, the goal of this work is to introduce coalescing reactions and thresholds to asymmetric DLAS. Informally speaking, our model features two particle types. Each site of a transitive unimodular graph has an $A$-particle with probability $p$, and otherwise a $B$-particle. $A$- and $B$-particles diffuse as simple random walks with possibly different jump rates. We allow particles of the same type to coalesce into clusters up to a given cap. When clusters of $A$- and $B$-particles collide, both are annihilated. Note that the our model is different than Stephenson's model because: (i) there are two particle types with possibly different jump rates, (ii) coalescing reactions occur up to, rather than beyond a threshold, and (iii) annihilation always occurs between unlike particles.
    
    This is a natural line of inquiry since coalescing random walk and annihilating random walk are closely related, but difficult to study together. Moreover, diffusion-limited interacting particles systems arose as models in physical chemistry \cite{chem1, chem2} for which coalescence is a natural reaction \cite{blythe2000stochastic}. The model is mathematically interesting, because, as we will see, the critical value and limiting particle density are difficult to infer when both annihilation and coalescence are present.

\subsection{Process Description}
	Fix a simple, locally finite, and connected graph $G = (\mathcal V,\mathcal E)$ with root $\0$. Each site $x \in \mathcal V$ initially contains $\xi_0(x)$ particles. If $\xi_0(x) >0$, then the particles are all of type $A_1$, and if $\xi_0(x) <0$, then the particles are of type $B_1$. 
 
 The site $x$ is associated with a tuple $(S^{x,j}, U^{x,j})_{j \in \mathbb Z}$. The $S^{x,j}$ are discrete simple random walk paths $S^{x,j} = (S_n^{x,j})_{n \geq 0}$ started at $x$ with Markov transition kernel $K$. So, $K$ is a function $K\colon\mathcal{V}\times \mathcal{V}\rightarrow [0,1]$ such that $K(u,v)=0$ if $(u,v)\notin \mathcal{E}$ and $\sum_{v\in N(u)} K(u,v)=1$ with $N(u)$ the set of vertices connected to $u$ by an edge.  
 The $U^{x,j}$ are independent \emph{braveries} sampled uniformly from $(0,1)$. The $j$th particle started at $x$ jumps along the vertices of $S^{x,j}$ according to a rate $\lambda_A>0$ exponential clock if it is an $A_i$-particle, and a rate $\lambda_B\geq 0$ clock if it is a $B_i$-particle for any $i \geq 1$. 

  To simplify the presentation, we consider initial configurations with either a single $A_1$-particle or single $B_1$-particle independently at each site. That is $(\xi_0(x))_{x \in \mathcal V}\in \{-1,1\}^{\mathcal V}$ is assigned according to a product measure with $\P(\xi_0(x) = 1)=p\in [0,1]$. So, our initial configuration has a $p$-density of $A_1$-particles and $1-p$ density of $B_1$-particles.

Fix \emph{caps} $ M,N \in [0,\infty]$. When multiple particles meet, the two particles with the highest bravery react. $A+B$ reactions result in mutual annihilation. When like particles meet they coalesce so long as both particles have not yet coalesced beyond the cap, otherwise, the particles do not interact and continue diffusing. These rules may be summarized as
\begin{equation} \label{eq:rules}
\begin{split}
    A_i + B_j &\to \varnothing \quad \forall i,j\geq 1, \\
    A_i + A_j &\to \begin{cases} A_{i+j},  &\max (i,j) \leq M \\  A_i\, A_j, & \text{else} \end{cases}\\
    B_i + B_j &\to \begin{cases} B_{i+j},  &\max (i,j) \leq N \\  B_i\, B_j, & \text{else}  \end{cases}.
\end{split}
\end{equation}

If multiple particles are present at the site, then the interactions are repeated pairwise giving priority to the highest bravery particles until no further interactions can occur at that instant. When a coalescing reaction occurs, the coalesced particle inherits the path and bravery of whichever of the two most recently coalesced particles has the highest bravery.
We refer to the index of coalesced particles as the \emph{size}. 
%
Call the interacting particle system with dynamics as at \eqref{eq:rules}  a \emph{diffusion-limited annihilating-coalescing system} (DLACS). We let $\xi_t^i = (\xi_t^i(x))_{x \in \V}$ denote the number of $A_i$ particles at $x$  at time $t$ for $i \geq 1$. Similarly, we set $\xi_t^{-i} = (\xi_t^{-i}(x))_{x \in \V}$ with $-\xi_t^{-i}(x)$ the number of $B_i$-particles at $x$ at time $t$. Set $\xi_t(x) = \sum_{i \in \mathbb Z} \xi_t^i(x)$. 

  \begin{figure}
   \centering 
   \includegraphics[height = 5.5 cm]{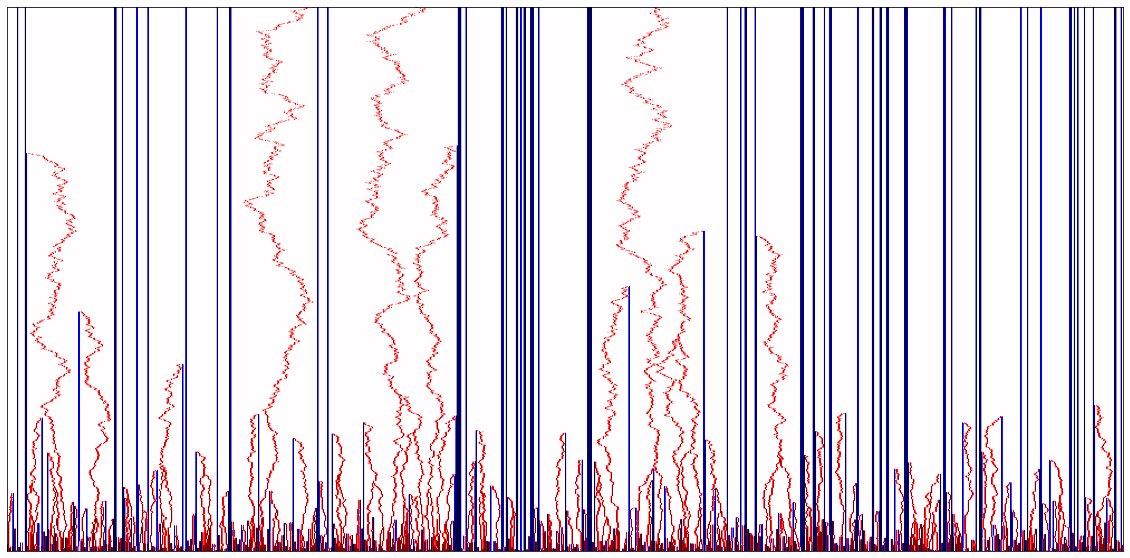}

   \includegraphics[height = 5.5 cm]{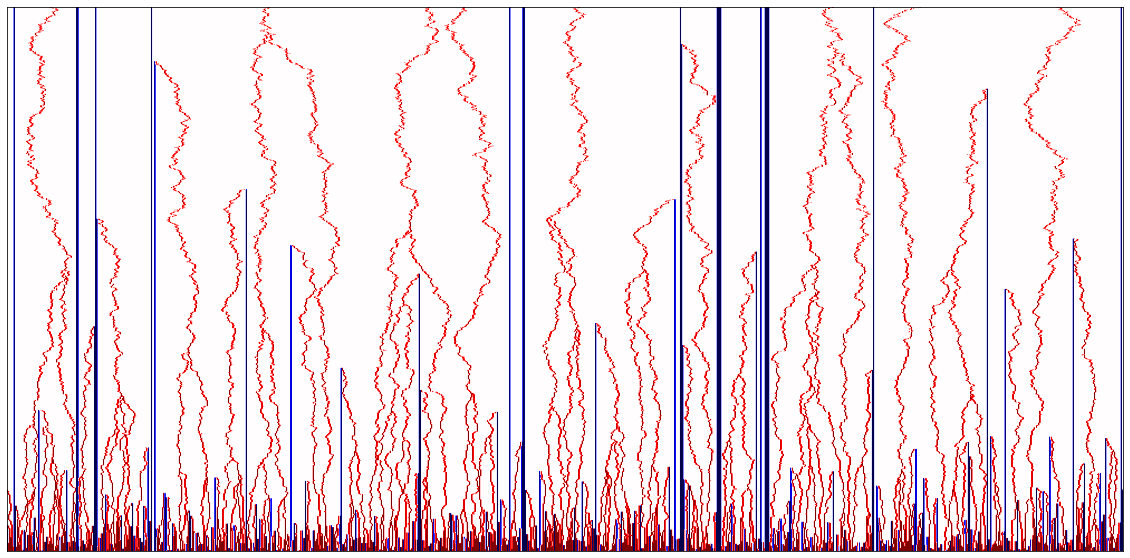}

   \includegraphics[height = 5.5 cm]{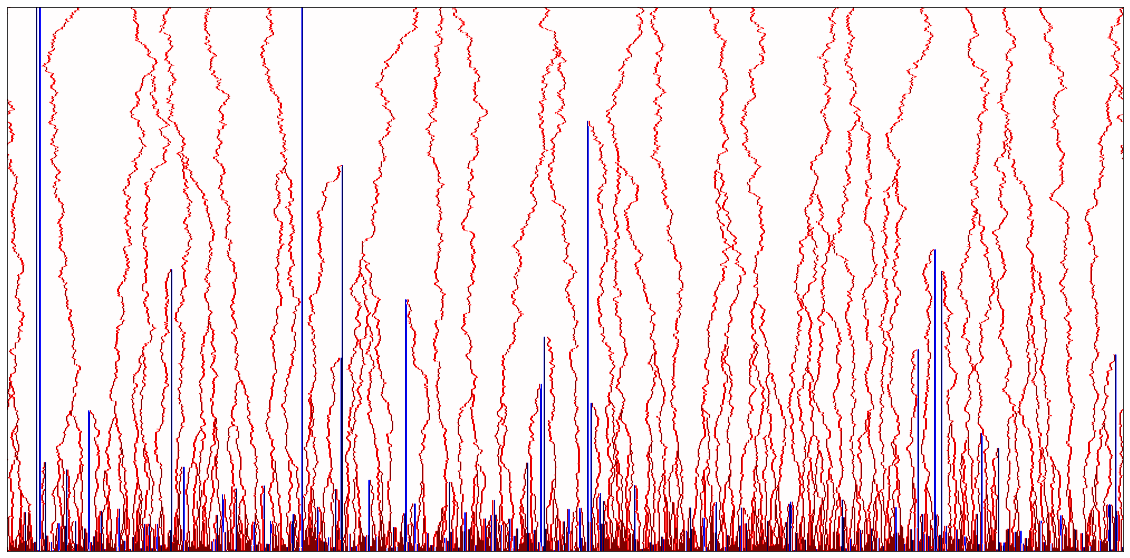}
   
   \caption{Simulations of DLACS with $\lambda_B=0$ and $M=\infty$ on a cycle with 2000 vertices. $A$-particles perform discrete time random walk for 2000 time steps. The top graphic has $p=.6$, the middle graphic has $p=.7$, and the bottom graphic has $p=.8$. Determining the critical value is an open problem.} \label{fig:sim}
\end{figure}
\subsection{Key quantities}

Viewing an $A_i$-particle as a mass of $i$-coalesced $A_1$-particles, we define
\begin{align}W_T(x) &= \int_0^T \sum_{i \geq 1} i\xi^i_t(x) \;dt \label{eq:WT}\end{align}
to be the \emph{weighted occupation time of $x$ by $A$-particles up to time $t$}. Instead, viewing each $A_i$-particle as a single particle, we define
\begin{align}
	V_T(x) &= \int_0^T \sum_{i \geq 1} \xi^i_t(x) \;dt \label{eq:VT}
\end{align}
to be the \emph{occupation time} of $x$ by $A$-particles up to time $t$.

We will denote the event that the vertex $u$ initially contains an $A_1$-particle by $\a_u$, and similarly for $\b_u$. In a convenient abuse of notation we will refer to the generic particle started at $u$ as  $\bullet_u$, and as $\a_u$ or $\b_u$ if its type is specified. We will view $A_i$-particles (and $B_i$-particles) as clusters of $i$ coalesced $A$-particles (or $i$ $B$-particles) that are following the path of the bravest particle in the cluster.  Define the events
\begin{align}
    \a_u \lrl \b_v &= \{\text{$\a_u$ and $\b_v$ mutually annihilate}\}\\
    \a_u \lrl \b &= \{\text{$\exists v$ such that $\a_u$ mutually annihilates with $\b_v$}\}\\
    \b_u \lrl \a &= \{\text{$\exists v$ such that $\b_u$ mutually annihilates with $\a_v$}\}.
\end{align}
We denote complements of these events with $\nlrl$. 

Let $\tau(\car_u)$ be the lifespan of an $A$-particle started at $u$, that is, the time at which the $A$-particle (and its coalesced cluster) mutually annihilates with a $B$-particle. Set $\tau(\car_u)=0$ if $u$ initially contains a $B$-particle, and $\tau(\car_u) = \infty$ if the $A$-particle at $u$ is never destroyed. Note that $\a_u$ is viewed as surviving in the event that it coalesces with another $A$-particle and begins following a different path. Let $S(\b_v)$ be the size of the particle that mutually annihilates with a $B$-particle at $v$. This is implicitly multiplied by $\ind{\b_v \lrl \a}$, so that $S(\b_v)=0$ if there is initially an $A$-particle at $v$, or if $\b_v \nlrl \car$. Define $S(\a_u)$ similarly, but for the size of the $B$-particle that collides with an $A$-particle initially at $u$. This lets us specify the sizes of particles at the time of mutual annihilation
\begin{align}
	\a_u \lrljk \b_v &= \{\text{$\a_u$ and $\b_v$ mutually annihilate}\} \cap \{S(\a_u) = k\} \cap \{S(\b_v) = j\}.
\end{align}

\subsection{Graph Topology}
We next describe conditions on the graphs for which our results apply. These conditions are necessary for a mass transport principle that we use to hold (\thref{lemma:mass_transport}). Similar hypotheses were used in \cite{cabezas2019recurrence, damron2019parking}. Canonical graphs such as integer lattices and regular trees satisfy these conditions. Denote by $\Aut(G)$ the group of all automorphisms of $G$. Let $\Aut_{K}(G)$ be the subgroup of $\Aut(G)$ consisting of all $K$-preserving automorphisms; that is 
\begin{equation*}
\Aut_{K}(G) = \{ \varphi\in \Aut(G) \colon K(u,v)=K(\varphi (u), \varphi (v) ) \quad \forall u,v\in \mathcal{V}  \}. 
\end{equation*}
Given a subgroup $\Gamma_{K}\le \Aut_{K}(G)$ of $K$-preserving automorphisms of $G$, for each $u,v\in \mathcal{V}$, denote $\Gamma_{K}(u,v)=\{ \varphi\in \Gamma_{K}\colon \varphi(u)=v \}$. We define the following conditions on the triple $(G,K,\Gamma_{K})$.
\begin{enumerate}[label = \arabic*.]
	\item (Transitivity)  $(G,K,\Gamma_{K})$ is \textit{transitive} if  $\Gamma_{K}(u,v)$ is nonempty for each $u,v\in \mathcal{V}$.
	\item (Unimodularity) $(G,K,\Gamma_{K})$ is \textit{unimodular} if for each $u,v\in \mathcal{V}$, 
	\begin{equation*}
		|\Gamma_{K}(u,u)v|=|\Gamma_{K}(v,v)u|<\infty.
	\end{equation*}
	\item (Infinite Accessibility)  For each $u, v \in \mathcal V$, we say that $u$ is accessible from $v$ if there exists a sequence $v = x_0, x_1, \hdots , x_n =u$ of adjacent nodes such that $\prod_{i=0}^{n-1} K(x_i
, x_i+1) > 0$. The triple $(G,K,\Gamma_{K})$ is \textit{infinitely accessible} if there exist $\varphi \in \Gamma_K$ and $u\in \mathcal{V}$ such that $\{\varphi^{n}(u)\colon n\ge 0\}$ is infinite and $u$ is accessible from $\varphi(u)$.
\end{enumerate}
Lastly, we say that $(G,K)$ has the \emph{random walk intersection property} if two independent random walks with jump rates $\lambda_B, \lambda_A >0$ collide almost surely regardless of their starting vertices.

\subsection{Results}
Note that all of the following theorems are for DLACS on a transitive, infinitely accessible, and unimodular graph.

We define the critical density as
	\begin{align}
	p_c &= p_c(\lambda_A, \lambda_B, M,N) := \sup\left \{ p \in [0,1] \colon  \f{p}{1-p} \f{ \E[S(\co) \mid \co \lrl \spot ] }{ \E[S(\so) \mid \so \lrl \car]} <1 \right\}.\label{eq:pc}
	\end{align}
 Intuitively, the quantity in the supremum for \eqref{eq:pc} balances the ratio of $A$- to $B$-particles with the ratio of the average size of $B$-particle and $A$-particle clusters when they are destroyed.

 Our first result establishes a phase transition for $A$-particle survival at $p=p_c$ for many cases. 
	

    


\begin{theorem}
    \thlabel{thm:main} 
Let $p_c$ be as defined at \eqref{eq:pc}. If either $N < \infty$, $\lambda_B=0$, or $(G,K)$ has the random walk intersection property, then
$$\P( \a_0 \lrl \b \mid \a_0) \begin{cases} =1, & p < p_c \\ <1, &p > p_c \end{cases}.$$
\end{theorem}
 



We also characterize the weighted occupation time of $\0$ by $A$-particles. This formula holds for any choices of $M,N,\lambda_A$ and  $\lambda_B$, however it only yields a linear lower bound in supercritical regimes for which $\P(\a_0 \lrl \b_0) >0$. In these regimes, the relationship $V_T \geq W_T / 2M$ also gives a linear growth bound on $V_T$ whenever $M < \infty$.
\begin{theorem} \thlabel{thm:WT}
For all $T \geq 0$
\begin{align}
\E[ W_T] = \int_0^T \P(\tau(\co) > t) dt \geq \P(\a_\0 \nlrl \b) T \label{eq:Vb}.
\end{align}
And, if $M < \infty$, then 
\begin{align}
\E[V_T] \geq \f {\P(\a_\0 \nlrl \b) } { 2M}T.
\end{align}
\end{theorem}

When $M=\infty$, there is no obvious comparison between $W_T$ and $V_T$. The DLACS with $p=1$ and $M=\infty$ is the classical coalescing random walk. It is proven in \cite[Theorem 1.2 (i)]{benjamini2016site} that  graphs with maximum degree $D$ satisfy $\E[V_T] \geq \log(1+DT)/D$. We prove an analogous logarithmic bound.
\begin{theorem} \thlabel{thm:VT}
Suppose that $M=\infty$. Let $D$ be the degree of the vertices in $G$. It holds that 
\begin{align}
\E \left[ \sum_{k \geq 1} \xi_t^k(\0) \right]\geq  \f{ [\P(\tau(\co) > t)]^2}{1+ 2Dt} \text{ for all $t \geq 0$}, \label{eq:density}
\end{align}
and thus,
$$\E [V_T] \geq \int_0^T \f{ [\P(\tau(\co) > t)]^2}{1+ 2Dt} dt 
\geq \f{[\P( \co \nlrl \spot)]^2}{2D}  \log(1+ DT)$$
for all $T \geq 0$.
\end{theorem}

As for an upper bound on $\E[V_T]$, the monotonicity result in \thref{lem:mono} ensures that $V_T$ is dominated by its behavior in the case $p=1$. This is the classical coalescing random walk for which the occupation time of the root grows logarithmically for trees and high dimensional lattices and more rapidly for graphs on which random walk is recurrent \cite{benjamini2016site, foxall2018coalescing, griffeath}. 

 We conclude by giving special attention to the symmetric case $M=N=\infty$ with equal rates $\lambda_A=\lambda_B=1$ and equal initial densities $p=1/2$. In this case, the dynamics can be summarized by the reactions $A+A\to A$, $B+B\to B$ and $A+B\to\emptyset$. Intuitively, these dynamics interpolate between the classical coalescing random walk, where particles always coalesce into a single particle upon meeting, and the annihilating random walk, where particles always mutually annihilate upon meeting.  This DLACS is equivalent to the case $n=2$ and $c_i = 1/2 = a_i$ from Stephenson's model \cite{stephenson1999asymptotic}. As discussed in the introduction, he was unable to analyze the case $n=2$. Thus, the upcoming \thref{thm:crw} resolves a missing case and does so in a robust way that holds for graphs beyond $\mathbb Z^d$.

 It is well known that the asymptotic density of particles in the annihilating system is $1/2$ that of the coalescing system \cite{arratia}. Naively, one might expect the DLACS model to split the difference and have $3/4$ the limiting particle density of the coalescing random walk. However, our final theorem shows that it is on the order of $2/3$. The reason for the $2/3$, given in our proof, is that there is a coupling that shows a particle is present in the DLACS model if and only if it is present in coalescing random walk and the coalescing history of the particle satisfies a certain property (see \eqref{eq:good}). We then show at \eqref{eq:ab} that the probability of satisfying this property has a recursive formulation with fixed point $2/3$. 

To state the theorem, let $\zeta_t(x)$ denote the coalescing random walk process. We set $\zeta_t(x) = 1$ if there is a particle at $x\in \mathcal V$ at time $t$ and $\zeta_t(x)=0$ otherwise. Initially $\zeta_0(x) =1$ for all $x$. Particles jump at rate $1$, and when a particle jumps it chooses a neighbor uniformly at random. When two particles meet, they coalesce into a single particle, which moves as a single random walk. 
 \begin{theorem} \thlabel{thm:crw}
     Suppose that $G = \mathbb Z^d$ or is a transitive unimodular graph on which the simple random walk is transient. If $M=N=\infty$, $\lambda_A=\lambda_B = 1$, and $p=1/2$, then as $t\to\infty$
     $$P(\xi_t(\0) \neq 0) \sim \frac23 P(\zeta_t(\0)=1)$$
 \end{theorem}

 The asymptotic behavior of $\P(\zeta_t(\0)=1)$ is known in this setting. For $\mathbb{Z}^d$, Bramson and Griffeath \cite{bramson1980asymptotics} proved that 
\begin{align}
\P(\zeta_t(\0) =1) \sim \begin{cases}
    \f 1 {\sqrt \pi} t^{-1/2}, & d =1 \\
    2 \pi \log t /t , & d =2 \\
    2 \gamma_d t^{-1}, & d \geq 3
\end{cases}
\label{eq:crwa}
\end{align}
with $\gamma_d$ the return probability for simple random walks started at the origin of $\mathbb Z^d$. For transitive unimodular graphs on which the simple random walk is transient, Hermon, Li,  Yao, and Zhang \cite{hermon2022mean} found that $\P(\zeta_t (\0) = 1) \sim \alpha t^{-1}$
with $\alpha$ the probability a random walk started at $\0$ and another started from a uniformly chosen neighbor of $\0$ do not meet.

\subsection{Further questions}
 It would be nice to resolve the missing case from \thref{thm:main}. The difficulty is that when $B$-particles coalesce with no cap the density of occupied sites decays rapidly. So, survival of $B$-particles is not enough to easily deduce that $A$-particles cannot indefinitely avoid contact with $B$-particles. 
\begin{question}
    Show that the survival probability of $A$-particles undergoes a phase transition at $p_c$ in the case $N=\infty$, $\lambda_B>0$, and $G$ does not have the random walk intersection property.
\end{question}

We would also like to know if critical values are continuous as the cap is increased.
\begin{question}
Fix $(G,K)$, $N$, $\lambda_A$, and $\lambda_B$ and let $p_c(M)$ be as at \eqref{eq:pc} with the convention that $p_c(\infty)$ is the case $M=\infty$. Is $\lim_{M \to \infty} p_c(M) = p_c(\infty)$?
\end{question}
The definition of $p_c$ in \eqref{eq:pc} is implicit. We are interested in finding explicit formulas. This is likely difficult for general $G$, but the graphs $\mathbb Z^d$ and the $d$-regular tree $\mathbb T_d$ are good places to start.
\begin{question}
Find $p_c$ for any graph and non-symmetric choice of $\lambda_A$, $\lambda_B$, $M$, $N$. 
\end{question}

An essential first step is proving that $p_c$ is non-trivial when $M=\infty$ and $N<\infty$. Here is a concrete question.

\begin{question}
Let $G = \mathbb Z^d$ and fix $\lambda_A=1$, $\lambda_B \in[0,\infty)$, $M=\infty$, and $N =0$. Is $p_c <1$?
\end{question}

Even showing that $p_c<1$ on $\mathbb Z$ with $\lambda_B=0$ would be interesting. This is similar to the setting from \cite{damron2019parking,johnson2020particle, parking_on_integers}, but with coalescence. See Figure~\ref{fig:sim} for simulations that suggest $p_c\in (.6,.8)$ in this case. A naive heuristic, inspired by \cite{droz1995ballistic}, is that the distance between two given $A$ particles is a rate-$2$ random walk, while the distance between a given $A$- and $B$-particle is a rate-$1$ random walk. Thus, on average $A+A \to A$ reactions should occur at twice the rate as $A+B \to \varnothing$ reactions. Across all reactions, $A$-particles disappear at twice the rate of $B$-particles. So, in this mean-field paradigm, we would expect the critical point to occur at $p=2/3$, when there are initially twice as many $A$- as $B$-particles. It is quite possible that a mean-field assumption is not justified as there are low-dimensional effects as occurred in similar systems \cite{bramson1991asymptotic}. Unpublished simulations from Brune suggest that $p_c > .7$, supporting the possibility that low-dimensional effects perturb $p_c$ away from mean-field predictions [private correspondence]. 

\subsection{Overview of proofs}
In \thref{lem:mono}, we prove that the longevity of all $A$-particles is non-decreasing as additional $A$-particles are introduced to the system. The argument uses tracer particles to track the difference between systems augmented with $A$-particles. Tracers were introduced in \cite{bramson1991asymptotic} and have been adapted in various ways \cite{cabezas2019recurrence, johnson2020particle, bahl2022diffusion}. 
The formula for $p_c$ in \thref{thm:main} and bounds in \thref{thm:WT} and \thref{thm:VT} are derived from different applications of the mass transport principle stated in \thref{lemma:mass_transport}. The lower bound on $\E[V_T]$ in \thref{thm:VT} uses an estimate on the particle size in coalescing random walk from \cite{foxall2018coalescing}. The proof of \thref{thm:crw} relies on a construction of DLACS for the given parameter values that uses the coalescing random walk. This coupling shows that in this symmetric case, the DLACS can be seen as a $2/3$-thinning of the coalescing random walk.

\section{Proofs of \thref{thm:main}, \thref{thm:WT}, and \thref{thm:VT}}

First we prove a monotonicity result for the lifespan of $A$-particles.

\begin{lemma} \thlabel{lem:mono}
Consider two systems $\xi$ and $\xi^+$ with identical underlying $(S^{x,j}, U^{x,j})_{j \in \mathbb Z}$ but $\xi_0(x) \leq  \xi_0^+(x)$ for all $x \in \mathcal V$. Define the lifespans of $\co$ in the two systems as $\tau(\co)$ and $\tau^+(\co)$. There exists a coupling such that $\tau(\co) \leq \tau^+(\co).$
\end{lemma}
\begin{proof}
        Following \cite[Lemma 1]{cabezas2019recurrence}, the quantity $\tau(\co)$ can be approximated by systems with finitely many particles in a finite ball centered at $\0$. Thus, it suffices to exhibit a coupling with $\tau(\co) \leq \tau^+(\co)$ in a system with finitely many particles ($\sum_{v \in \V} |\xi_0(v)| < \infty$) and
        $\xi^+$ introducing an $A_1$-particle at a site $x$. 

        Suppose that $\xi^+(x) = j^*$. We modify the bravery $U^{x,j^*}$ to be the minimum of all other braveries in the (finite) system divided by $2$. Following \cite[Section 3.1]{cabezas2019recurrence}, we will account for the differences introduced by an extra $A_1$-particle at $x$ with a \emph{tracer} that can be either \emph{active, dormant,} or \emph{dead}. We say that the tracer is tracking a given particle if it is following the tracked particle's random walk path. 
  
         The tracer initially is \emph{active} and tracks the extra $A$-particle at $x$ following $S^{x,j^*}$ when $j^*\geq 1$. The state of the tracer changes depending on the size/type of particle it is tracking. 
       
        \begin{description}
            \item[Tracking an $A$-particle with size $\leq M$] \quad 
                \begin{itemize} 
                    \item Suppose that the tracked particle coalesces with another $A$-particle. If the resulting size is $\leq M+1$, then the tracer becomes/remains dormant and continues tracking the coalesced particle. If the resulting size is $>M+1$, then the tracer dies. 
                    \item Suppose that the tracked particle mutually annihilates with a $B$-particle. If the tracer is active, then it remains active and begins following the path of the $B$-particle. If the tracer is dormant, then it dies.
                \end{itemize}
            \item[Tracking an $A$-particle with size $M+1$]\quad 
                \begin{itemize} 
                    \item  Suppose that the tracked particle meets an $A$-particle with size $\leq M$. The particles do not coalesce. The tracer becomes/remains active and tracks whichever of the two non-coalesced particles has lower bravery. 
                    \item Suppose that the tracked particle mutually annihilates with a $B$-particle. If the tracer is active, the tracer remains active and begins following the path of the $B$-particle. If the tracer is dormant, then it dies.
                \end{itemize}
            \item[Tracking a $B$-particle] If the tracer is active and meets another $A$-particle, then it begins tracking that particle and remains active. The size of the $A$-particle does not increase. If the tracer is dormant, then it dies.
        \end{description}
    The tracer does not influence the evolution of $\xi^+$, rather, it accounts for the discrepancy between $\xi$ and $\xi^+$. If the tracer is active and tracking an $A$-particle, then it accounts for an extra $A$-particle in $\xi^+$. If it is active and tracking a $B$-particle, this accounts for an extra $B$-particle in $\xi$. Otherwise, the positions of the particles in $\xi$ and $\xi^+$ are identical. This is described for non-coalescing systems in \cite[Section 3.1]{cabezas2019recurrence}. Thus, if $\co$ never interacts with the tracer, then the extra particle has no effect on $\tau^+(\co)$, so  $\tau(\co)= \tau^+(\co)$. If $\co$ is at some point tracked by the tracer or interacts with a tracked particle, the tracer rules are such that $A$-particle paths are extended. Thus, $\tau(\co) \leq \tau^+(\co)$. This gives a coupling with the claimed inequality. 
\end{proof}

Our main tool is a mass-transport principle. The following is a minor modification of \cite[Theorem 8.7]{lyons2016probability}. 

\begin{lemma}\thlabel{lemma:mass_transport}
	Let $Z\colon \mathcal{V}\times \mathcal{V}\rightarrow [0,\infty)$ be a collection of random variables such that $\E [Z(u,v)] = \E [Z(\varphi(u),\varphi(v))]$ for all $\varphi\in\Gamma_K$ whenever $v$ is accessible from $u$ or $u$ is accessible from $v$, and $\E[ Z(u,v)] = 0$ otherwise.
	Then 
	\begin{equation*}
	\E   \sum_{v \in  \mathcal{V} } Z(\root ,v) = \E   \sum_{u \in \mathcal{V}  } Z(u,\root). 
	\end{equation*}
\end{lemma}

\begin{lemma}\thlabel{prop:gamma_MTP}For all $p \in [0,1]$ it holds that
	 \begin{align}
	 \Eco &= \Eso 
  \label{eq:mtp}.
	 \end{align}

\end{lemma}

\begin{proof}

For each $u,v\in \mathcal{V}$ and pair of integers $j,k\ge 1$, define an indicator variable 
	\begin{align}
		Z_{j,k}(u,v) = \ind{\car_{u} \wedge \spot_{v}\wedge  (\car_{u}\lrljk \spot_{v})}.
	\end{align}
	\thref{lemma:mass_transport} ensures that
	\begin{align}
	k\P(\car_\0 \lrljk \spot) =\E \sum_{v\in\mathcal V} Z_{j,k}(\0,v) = \E\sum_{u\in\mathcal V}Z_{j,k}(u,\0)= j \P( \car \lrljk \so ).
	\end{align}
	We have used the fact that there are exactly $k$ and $j$ distinct sites corresponding to the $k$- and $j$-coalesced particles that are counted by the indicators. Summing over all $j,k\ge 1$ gives \eqref{eq:mtp}.
\end{proof}



\begin{proof}[Proof of \thref{thm:main}]
By conditioning, we may expand \eqref{eq:mtp} as 
$$\E[S(\co) \mid \co \lrl \b ] \P(\co \lrl \b \mid \co) p =\E[S(\so) \mid \so \lrl \a ] \P(\so \lrl \a \mid \so)(1-p) .$$
Rearranging gives
\begin{align}
\f{p}{1-p}\f{\E[S(\co) \mid \co \lrl \spot ]}{\E[S(\so) \mid \so \lrl \car ]}  = \f{ \P(\so \lrl \car \mid \so) }{ \P(\co \lrl \spot\mid \co )}. \label{eq:condition}
\end{align}
Thus, whenever 
$$ \f{p}{1-p} \f{\E[S(\co) \mid (\co \lrl \spot) \wedge \co ]}{\E[S(\so) \mid (\so \lrl \car) \wedge \so ]} >1,$$
we have $\P(\co \lrl \b \mid \co) <1$. 

On the other hand, when
$$ \f{p}{1-p} \f{\E[S(\co) \mid (\co \lrl \spot) \wedge \co ]}{\E[S(\so) \mid (\so \lrl \car) \wedge \so ]} <1,$$
we have 
$$ \f{ \P(\so \lrl \car \mid \so )}{ \P(\co \lrl \spot\mid \co )} <1.$$
This implies that $\P(\b_\0 \lrl \a \mid \b_\0) <1$.

The conclusion that $\P(\b_\0 \lrl \a \mid \b_\0) <1$ when $p< p_c$ ensures that $\P(\b_\0 \nlrl \a) >0$. Thus, there is a positive density of $B$-particles that survive for all time. The assumption $N<\infty$ ensures that the density of distinct $B$-particles is least $\P(\b_\0 \nlrl \a) / 2N$. If $N=\infty$, but $\lambda_B=0$, we also have a positive limiting density of $B$-particles since no $B$-particle coalescence occurs. 

By similar reasoning as \cite[Lemma 6]{cabezas2019recurrence}, a positive density of distinct $B$-particles ensures that an independent rate $\lambda_A$ random walk $(X_t)$ started at $\0$ will almost surely coincide with a site containing  $B$-particle. This implies that $\P(\a_0 \lrl \b \mid \a_0) = 1$. If $N=\infty, \lambda_B>0$, and $(G,K)$ has the random walk intersection property, then almost surely the path of $X_t$ will intersect with the path initially assigned to each $B$-particle in the system (extending the path beyond the lifespan of the $B$-particle). Since some $B$-particles survive forever when $p<p_c$, it follows that $X_t$ will eventually collide with a $B$-particle almost surely. 

It follows from \thref{lem:mono} that $\P(\a_0 \nlrl \b \mid \a_0) = \P(\tau(\co) = \infty \mid \co)$ is increasing in $p$. Thus, for all $p< p_c$ we have $\P(\a_0 \nlrl \b \mid \a_0)=0$ and for all $p >p_c$ we have $\P(\a_0 \nlrl \b \mid \a_0)>0$. 
\end{proof}

\begin{proof}[Proof of \thref{thm:WT}]
		To obtain \eqref{eq:Vb}, let $L_t(\car_u)$ denote the location of $\car_u$ at time $t$ if $\car_u$ is still in the system at time $t$. Define the family of indicator random variables $W(u,v,t) = \ind{ L_t(\car_u) = v }$ for $u,v \in \mathcal V$ and $t \geq 0$. \thref{lemma:mass_transport} gives
\begin{align}
\P(\tau(\co) > t) = \E \sum_{v\in  \mathcal V} W(\0,v,t) = \E \sum_{u\in  \mathcal V} W(u,\0,t) = \E \sum_{i \geq 1} i\xi_t^i(\0).\label{eq:tau}
\end{align}
Integrating gives the claimed formula.
\end{proof}

\begin{proof}[Proof of \thref{thm:VT}]
Let $N_{t,k} = \xi_t^k(\0)$ be the number of size-$k$ $A$-particles at $\0$ at time $t$, $N_t = \sum_{k=1}^\infty N_{t,k}$ be the total number of $A$-particles at $\0$ at time $t$, and $n_t = \sum_{k=1}^\infty k\xi_t^k(\0)$ be the weighted number of $A$-particles at $\0$ at time $t$.
Let $\texttt{size}_t(\a_u)$ denote the size of the cluster that $\a_u$ belongs to at time $t$ with the convention that it is zero if $\a_u$ is no longer in the system.  For $u,v \in \V$, $t \in [0,\infty]$ and $k \geq 1$ define the indicators
$$R(u,v,t,k) = \ind{L_t(\a_u) = v, \texttt{size}_t(\a_u) = k}.$$ 

By \thref{lemma:mass_transport}
\begin{align}
    \P(\tau(\a_\0) > t, \texttt{size}_t(\a_\0) = k)  &= \E \sum_{v\in  \mathcal V} R(\0,v,t,k) = \E \sum_{u\in  \mathcal V} R(u,\0,t,k) = \E[k N_{t,k}]. 
\end{align}
Multiplying both sides by $k$ and summing gives
\begin{align}
\E[\texttt{size}_t(\a_\0)] = \sum_{k=1}^\infty k^2\E [N_{t,k}].\label{eq:size=}
\end{align}

Let $D$ be the degree of vertices in $G$. If follows from \cite[Proposition 2.4]{foxall2018coalescing} and \thref{lem:mono} that $\E[\texttt{size}_t(\a_\0)] \leq 1 + 2 D t.$ Applying this to \eqref{eq:size=} gives
\begin{align}
	\E \sum_{k=1}^\infty k^2 N_{t,k} \leq 1 + 2 Dt \label{eq:k2}.
\end{align}

Writing $\sum_{k=1}^\infty k N_{t,k} = \sum_{k=1}^\infty k \sqrt{N_{t,k}}\sqrt{N_{t,k}}$ and applying the Cauchy-Schwartz inequality yields
	\begin{align}
		n_{t}^2 \le \left( \sum_{k=1}^{\infty} k^{2}N_{t,k}\right) \left(\sum_{k=1}^{\infty}  N_{t,k}\right) = \left( \sum_{k=1}^{\infty} k^{2}N_{t,k}\right) N_{t}.
	\end{align}
Taking expectation and then applying the bound at \eqref{eq:k2} gives
	\begin{align}
		\E[n_{t}^{2}] \leq \E[N_{t}] \left( \sum_{k=1}^{\infty} k^{2}\E[N_{t,k}]\right) \leq \E[N_t] (1+ 2Dt).
	\end{align}
Since $(\E[n_t])^2 \leq \E[n_t^2]$, we may rearrange the above inequality to obtain $\E[N_t] \geq (\E[n_t])^2/(1+2Dt).$
It is proven at \eqref{eq:tau} that $\E[n_t] = \P(\tau(\co) > t)$. This gives the claimed inequality for $\E[N_t]$.
\end{proof}

\section{Proof of \thref{thm:crw}}

We first describe a coupling between the coalescing random walk $\zeta_t$ and DLACS $\xi_t$, which gives an alternate construction of DLACS in the case $M=N=\infty$, $\lambda_A=\lambda_B=1$, and $p=1/2$. Our probability space consists of a collection of independent compound Poisson processes, one for each oriented edge in the graph $G$ (for unoriented edges, we assign a Poisson process for each orientation) and a collection of i.i.d.~$\pm 1$-valued random variables, one for each vertex in $\mathcal V$. The Poisson process on the edge $(u,v)$ has rate $1/\deg(u)$, and at the times of this process we put an arrow from $u$ to $v$. In addition, we mark each arrow independently with an \texttt{OR} or a \texttt{XOR} with equal probability.

The coalescing random walk $\zeta_t \in \{0,1\}^{\mathcal V}$ indicates whether or not a coalesced particle is at each site at time $t$. It starts with a single particle at each vertex of $\mathcal V$. A particle at $u$ will move the first time it encounters an arrow from $u$ to any neighbor of $u$, and it moves to the indicated neighbor. If there is already a particle at the neighbor, then the particles coalesce and continue following arrows together forevermore. The marks of \texttt{OR} and \texttt{XOR} are ignored by the coalescing random walk. See Figure~\ref{fig:CRW} for a depiction of the  graphical  construction of the coalescing  random walk on a subset  of $\mathbb{Z}$ -- the marks \texttt{OR}  and \texttt{XOR} are omitted.

  \begin{figure}
   \centering 
   \includegraphics[width = .5\textwidth]{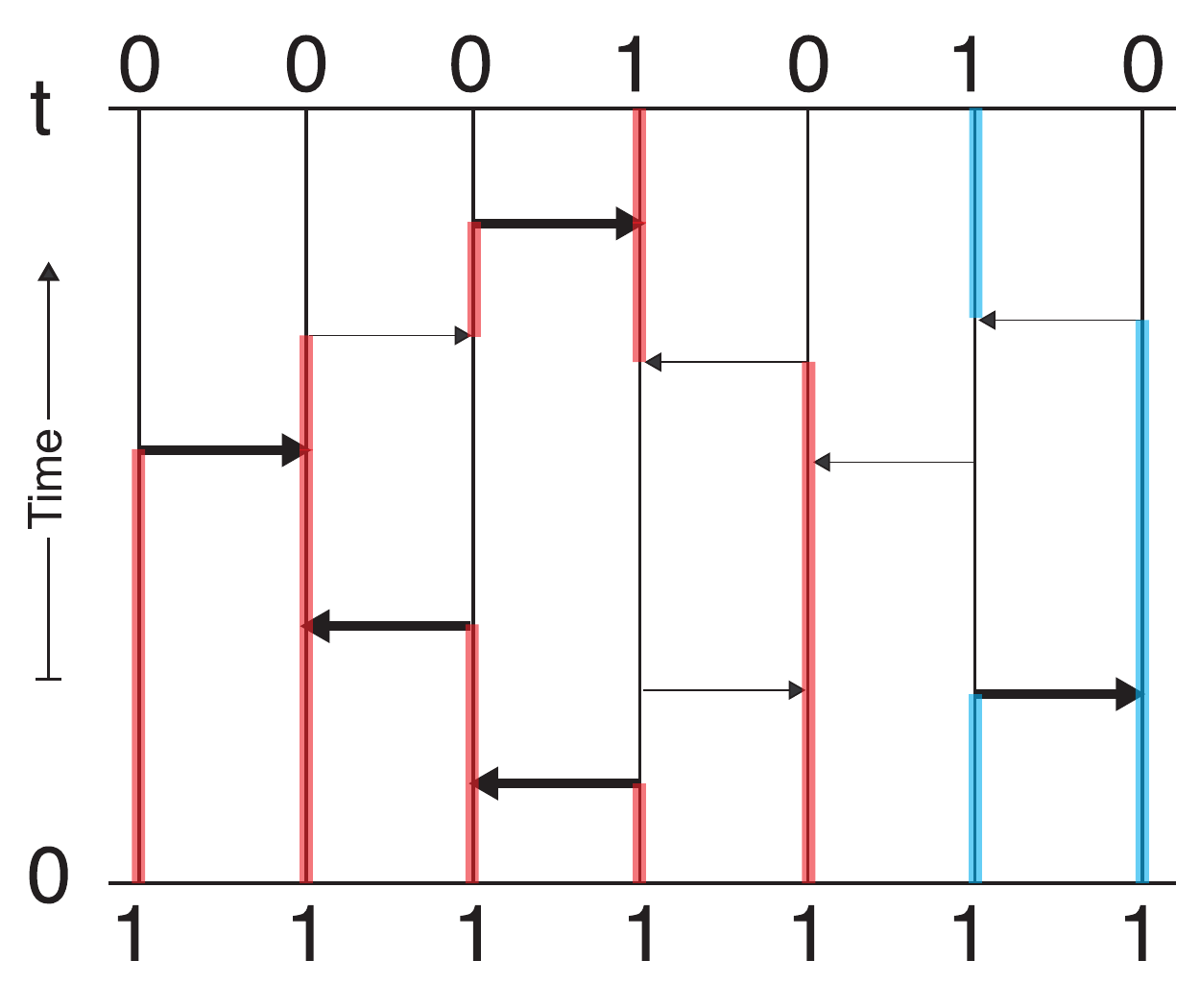}
   
   \caption{Graphical  construction  of the coalescing random walk $\zeta_t$ on a subset of $\mathbb{Z}$. Arrows along each oriented edge occur at times of independent Poisson processes with rate $1/2$. Particles move from the tail  to  the head of  an arrow, and when  two particles meet they coalesce. Arrows in bold signify coalescence events. Bold (red or blue)  lines indicate the presence of particles.} \label{fig:CRW}
\end{figure}

To construct the DLACS, we again start with one particle per site. We know that these particles are equally likely to be of types $A$ or $B$ (since $p=1/2$), but we will \textit{not} reveal these labels. Instead, observe that when two particle meet, they should coalesce if their types are the same ($AA$ or $BB$) and annihilate if their types are not the same ($AB$ or $BA$), and these happen with equal probability when two particles encounter each other for the first time. Thus, in our construction, when a particle jumps to a new location that already contains another particle, the two particles annihilate if the arrow is labeled \texttt{XOR}, and they coalesce if the arrow is labeled \texttt{OR}. In the former case, both particles are removed, and no longer interact. In the latter case, the two particles continue following arrows together as though they were one particle. Note that in this case, conditional on the particles coalescing, we know only that they were of the same type, but the coalesced particle is still equally likely to be of type $A$ or $B$. When a particle moves to a location that does not contain another particle, then it simply  moves there regardless of the label (\texttt{OR} or \texttt{XOR}) on the arrow. 

  \begin{figure}
   \centering 
   \includegraphics[width = .45\textwidth]{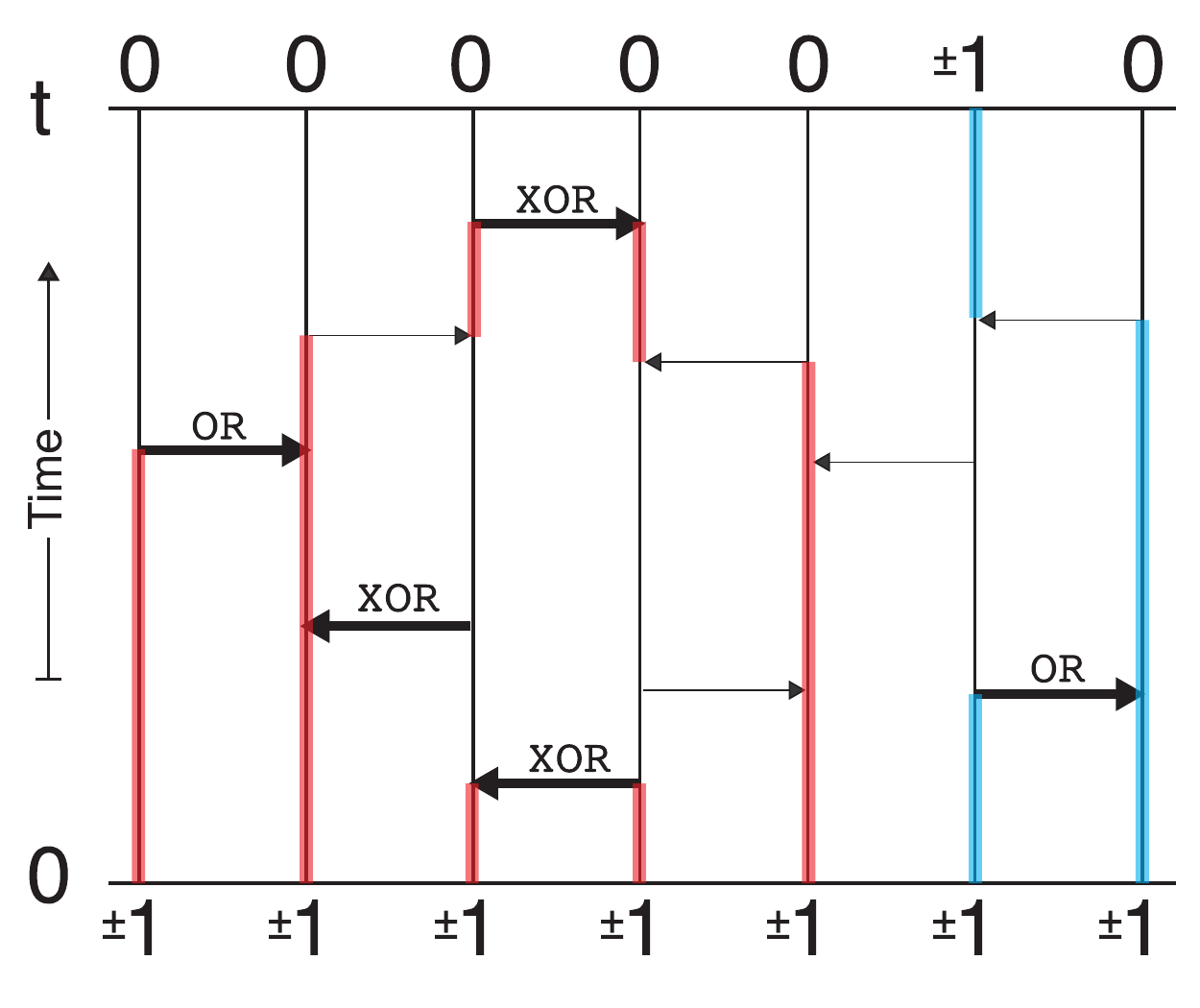} 
   \hfill
  \includegraphics[width = .45\textwidth]{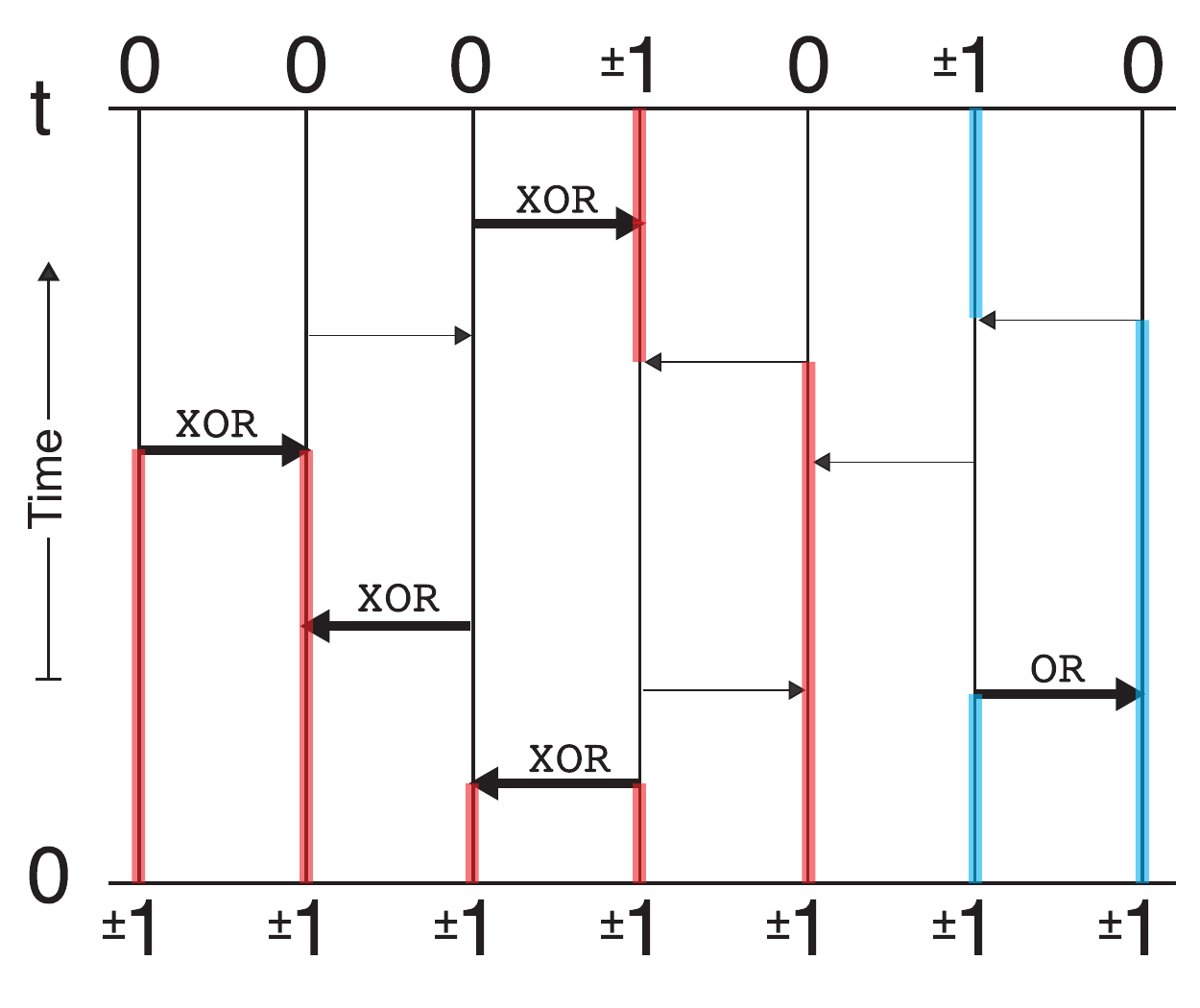}
   
   \caption{Graphical  constructions of the DLACS $\xi_t$ on a subset of $\mathbb{Z}$. The arrows are the same as in Figure~\ref{fig:CRW}, and only the labels of the bold arrows are shown, as the labels on the remaining arrows are irrelevant. Bold (red or blue)  lines indicate the presence of particles. The difference  between the  left and right realizations is the label  of  the leftmost arrow; note the effect of this label  on the presence of the particle at the  middle site.} \label{fig:DLACS}
\end{figure}

By considering finitely many particles together and tracking their trajectories, it follows by induction along the jump times that the types of the particles remaining at time $t$ are independent and equally likely to be $A$ or $B$. Thus, at a fixed time $t$, the types of the particles that remain are independent and equally likely to be  $A$ or $B$. 
This gives $\xi_t$ the correct marginal distribution. 

Conditional on $\{\zeta_t(\0) =1\}$, the trajectories followed by the coalesced particles at $\0$ induce a binary tree $\T_t$ whose vertices correspond to coalescence events. We let $\Ts_t$ be the labeled version of $\T_t$ in which each vertex is given the same label $\OR$ or $\XOR$ as the arrow from the graphical construction that induced the collision. We say that $\Ts_t$ is \emph{good} if the labels are such that, in this coupling, there is a particle present at $\0$ in DLACS. It follows from the construction and this definition that
\begin{align}
\P(\xi_t(\0) \neq 0) = \P(\Ts_t \text{ is good}) \P(\zeta_t(\0) =1).\label{eq:good}
\end{align}

  \begin{figure}
   \centering 
   \includegraphics[width = .43\textwidth]{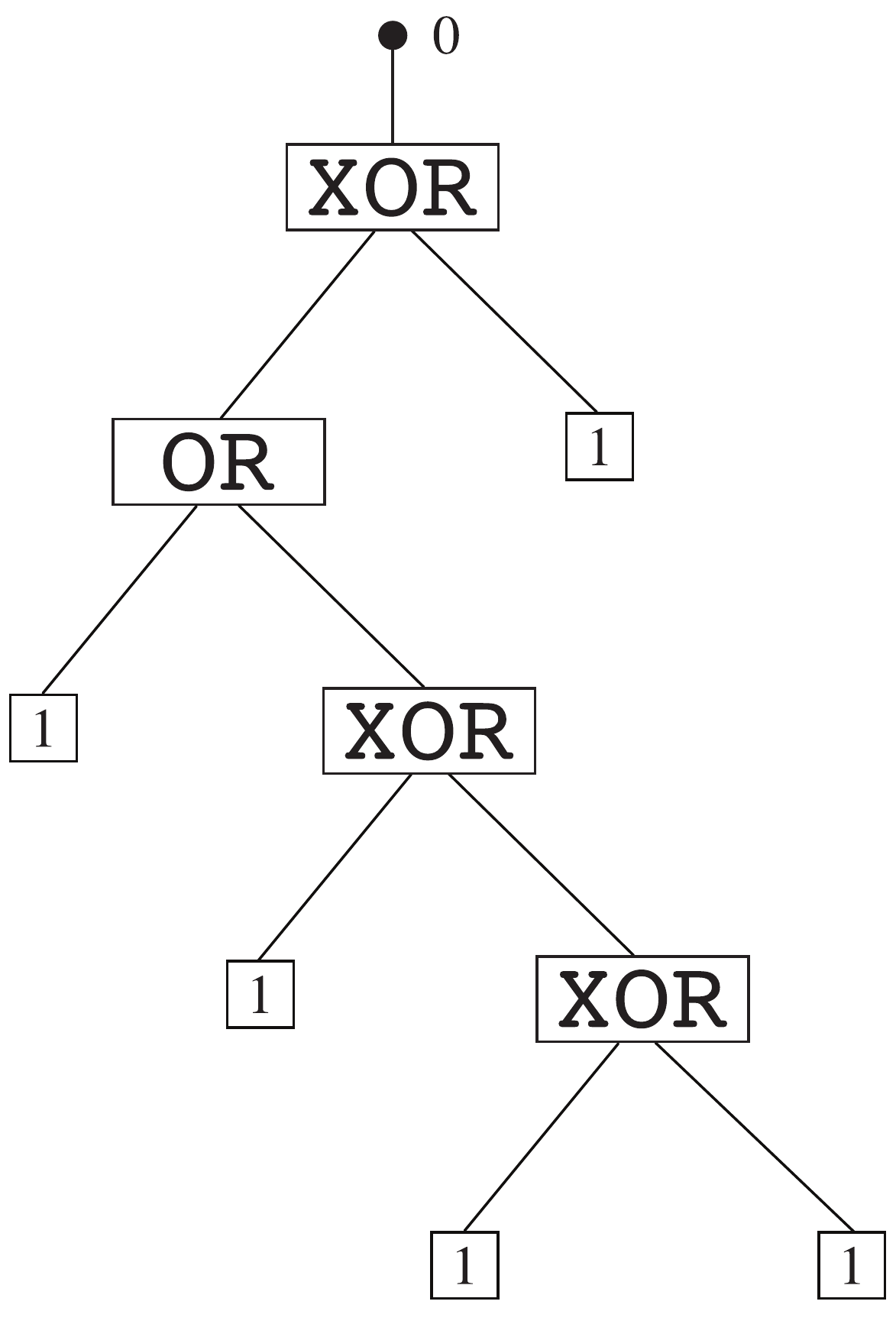} 
    \hfill
  \includegraphics[width = .43\textwidth]{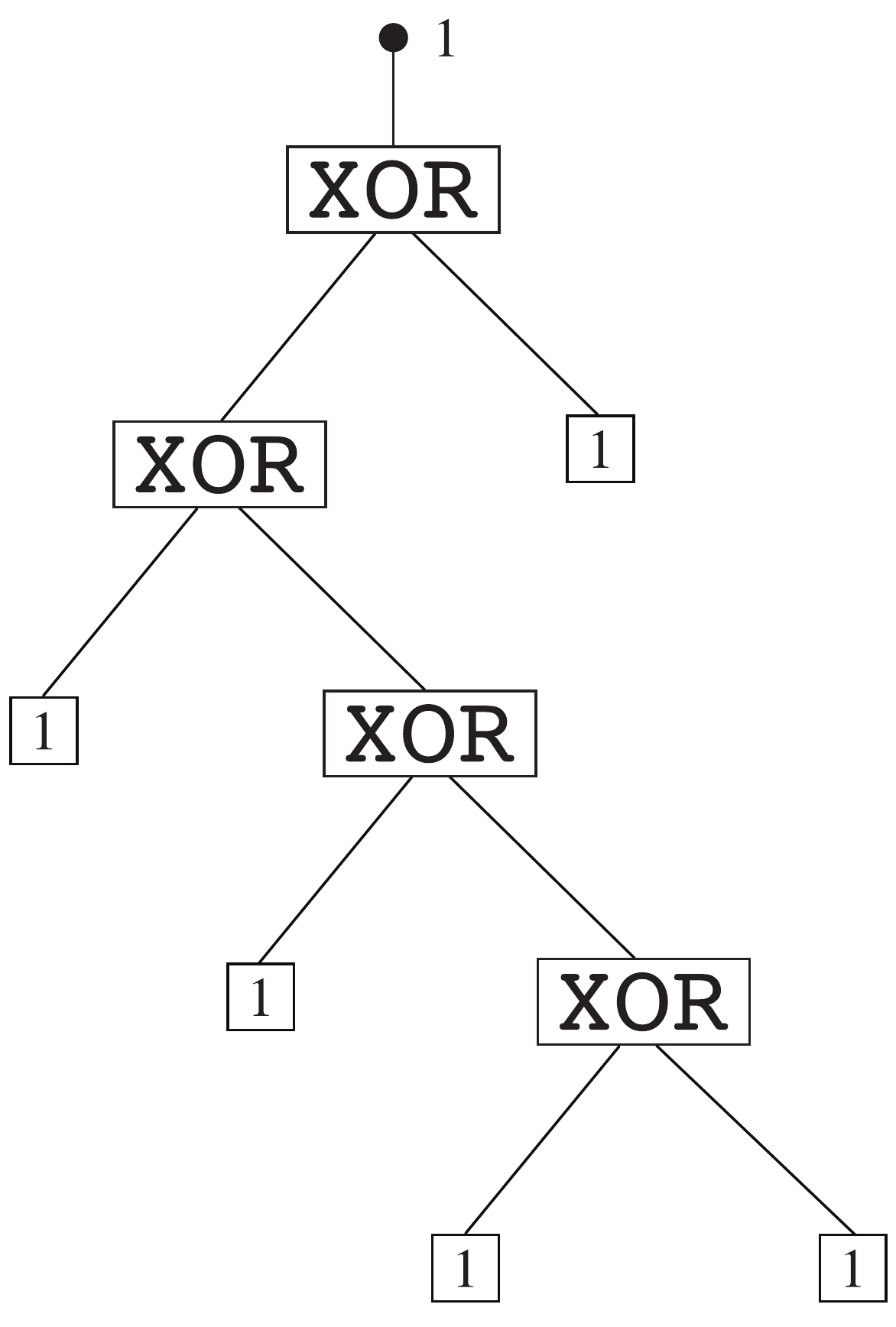}
   
   \caption{The dual trees $\Ts_t$ corresponding to the middle vertex (red component) in Figure~\ref{fig:CRW} with labels from Figure~\ref{fig:DLACS}. The nodes act as logical \texttt{OR} and \texttt{XOR} gates with inputs of $1=\texttt{True}$ at the leaves and output $0$ or $1$ at the root.} \label{fig:dual-tree}
\end{figure}

Let $|\Ts_t|$ be the number of leaves in $\Ts_t$. This corresponds to the number of coalesced particles at $\0$ at time $t$. First note that $\P(\Ts_t \text{ is good} \mid |\Ts_t| = 1) = 1$ since there are no collisions. Now suppose $|\Ts_t| = n\ge 2$. Following the trajectory of the particle at $\0$ in $\zeta_t$ backwards in time, there is a most recent time at which it coalesced with another particle, say at time $0<s\le t$ along an arrow from $u$ to $v$. There are two cases depending on the label of the arrow at $(u,v)$ crossed at time $s$. If the label is \texttt{OR} and there is at least one particle at $u$ or at $v$ in $\xi_s$, then there will be a particle at $\0$ in $\xi_t$. If the label is \texttt{XOR} and there is \textit{either} a particle at $u$ or a particle at $v$ (but not both) in $\xi_s$, then there will be a particle at $\0$ in $\xi_t$.

Suppose the conditional probability (given $|\Ts_t|$ and the history of $\zeta_t$) that there is a particle at $u$ in $\xi_s$ is $\alpha$ and the corresponding probability for $v$ is $\beta$. Observe that these events are conditionally independent, as they depend only on the labels of arrows along disjoint coalescing random walk histories. It follows that 
\begin{align}
\P(\Ts_t \text{ is good } \mid |\Ts_t| = n) &=\frac12[1 - (1-\alpha)(1-\beta)] + \frac12[\alpha(1-\beta)+\beta(1-\alpha)]\\
&= \alpha+\beta - \frac32\alpha\beta. \label{eq:ab}
\end{align}
Note that it is easy to check that this probability is minimized on $[1/2,1]^2$ when $\alpha=\beta=1$, which gives the lower bound 
\begin{align}
    \P( \Ts_t \text{ is good} \mid |\Ts_t| = n) \geq \f 12 \label{eq:12}
\end{align}
for all $n \geq 1$. This further implies that
\begin{align}
    \P( \Ts_t \text{ is good}) \geq \f 12.
\end{align}



\begin{lemma} \thlabel{lem:size}
    $|\P(\Ts_t \text{ is good} \mid |\Ts_t| \geq k ) - \f 23| \leq k^{-1}$. 
\end{lemma}

\begin{proof}
    Since $\Ts_t$ is a binary tree, if $|\Ts_t| \geq k$ then there is a leaf $v$ with distance $m \geq \log_2 k$ to the root of $\Ts_t$.
    Start at $v$ and work upwards. Suppose that the probability of a particle surviving after the $k$th gate is $q$ and at the $(k+1)$st gate, suppose the probability of there being a particle along the other (adjoining) path is $r$. Then the probability of a particle persisting after the $(k+1)$st gate is, by \eqref{eq:ab}, 
    $$q + r - \f 32 qr.$$ 
    This probability is closer to $2/3$ by a factor of at least $1/2$ than $q$. This follows because
$$\left |q+r - \f32 qr - \f 23\right| = \left |1- \f 32 r\right|\times \left|q - \f 23 \right|.$$
And since \eqref{eq:12} implies that $r \in [1/2,1]$, we must have $|1-3r/2| \leq 1/2$. Therefore, the probability $\hat p$ that the particle corresponding to $v$ is still present at the root, satisfies $|\hat p-2/3| < 2^{-m} \leq 2^{-\log_2 k} = k^{-1}$.
\end{proof}

\begin{lemma} \thlabel{lem:sizeT}
    Suppose that $G= \mathbb Z^d$ or is a transitive unimodular graph on which the simple random walk is transient. If $k>0$ and $\epsilon >0$ are fixed, then $\P(|\Ts_t| \leq k ) < \epsilon$ for all sufficiently large $t$. 
\end{lemma}


\begin{proof}
Note that $|\Ts_t|$ has the same law as the number of coalesced particles at $\0$ conditional on the event that $\0$ is occupied. This has the same law as the number of sites $n_t$ at time $t$ with the opinion started from $\0$ in the voter model. To define the voter model, each vertex $x$ is assigned a unique number,
which we think of as its opinion. Then for any voters $x$ and $y$ sharing an edge, with rate $1$ the voter $y$ adopts
the opinion of $x$. Let $\zeta_t^{\0}$ be the set of vertices at time $t$ with the opinion initially held by $\0$. See \cite{benjamini2016site} or \cite{bramson1980asymptotics} for a definition of the voter model and description of the duality relationship. 

We start with the case $G = \mathbb Z^d$. \cite[Introduction and Theorem 1]{bramson1980asymptotics} give that for $G= \mathbb Z^d$ and any $\alpha \in (0,\infty)$ we have 
\begin{align}
\P(n_t \leq \alpha \E[n_t \mid n_t>0] \mid n_t >0) \to 1- e^{-\alpha} \text{ as $t \to \infty$}. \label{eq:exp}
\end{align}
\cite[Theorem 1.15]{hermon2022mean} and the same reasoning used to derive \cite[(19)]{bramson1980asymptotics} give that \eqref{eq:exp} holds when $G$ is a transitive unimodular graph on which random walk is transient.
Additionally, \cite[Theorem 1]{bramson1980asymptotics} and \cite[Theorem 1.12]{hermon2022mean} give that $\E[n_t\mid n_t>0] \to \infty$ for the $G$ in our hypotheses.
It follows that for $\alpha = -\log(1-\epsilon/2)$
$$\limsup_{t\to\infty} \P(|\Ts_t| \leq k ) = \limsup_{t\to\infty}\P(n_t \leq k \mid n_t >0) \leq 1- e^{-\alpha} = \epsilon/2,$$
so $\P(|\Ts_t| \leq k )< \epsilon$ for all large $t$.

\end{proof}

\begin{proof}[Proof of \thref{thm:crw}]
Using \eqref{eq:good}, it suffices to prove that $\P(\Ts_t \text{ is good}) \to 2/3$ as $t \to \infty.$ 
To this end, fix $\epsilon \in (0,1)$. Let $k>0$ be such that $k^{-1} < \epsilon$.  For any $t\geq 0$ we may write 
\begin{align}
\P(\Ts_t \text{ is good} ) &= \P( \Ts_t \text{ is good} \mid |\Ts_t| > k ) \P(|\Ts_t| > k ) \\
& \hspace{ 2 cm} + \P( \Ts_t \text{ is good} \mid |\Ts_t| \leq k ) \P(|\Ts_t| \leq k ). \label{eq:Tg}
\end{align}

By \thref{lem:size} and \thref{lem:sizeT}, the first term on the right of \eqref{eq:Tg} lies in the interval $\f 23 (1 - \epsilon) \pm \epsilon$
for all large $t$. Similarly, the second term on the right of \eqref{eq:Tg} lies in the interval $[0,\epsilon]$. Taking $\epsilon \to 0$, we see that
$\lim_{t \to \infty} \P(\Ts_t \text{ is good}) = 2/3$ as desired.  
\end{proof}

\bibliographystyle{amsalpha}  	
\bibliography{mybib}

\end{document}